\label{key}
\documentclass[10]{article}
\usepackage{amsfonts,amssymb,amsmath,amsthm,cite}
\usepackage{graphicx}

\usepackage[applemac]{inputenc}
\usepackage{amsmath,amssymb,amsthm, hyperref, euscript}
\usepackage[matrix,arrow,curve]{xy}
\usepackage{graphicx}
\usepackage{tabularx}
\usepackage{float}
\usepackage{hyperref}
\usepackage{tikz}
\usepackage{slashed}
\usepackage{mathrsfs}

\usetikzlibrary{matrix}

\usepackage[T1]{fontenc}
\usepackage{amsfonts,cite}
\usepackage{graphicx}

\usepackage[applemac]{inputenc}

\usepackage[sc]{mathpazo}
\DeclareFontFamily{T1}{pzc}{}
\DeclareFontShape{T1}{pzc}{m}{it}{1.8 <-> pzcmi8t}{}
\DeclareMathAlphabet{\mathpzc}{T1}{pzc}{m}{it}

\theoremstyle{plain}
\newtheorem{prop}{Proposition}[section]

\newtheorem{lem}[prop]{Lemma}
\newtheorem{cor}[prop]{Corollary}
\newtheorem{thm}[prop]{Theorem}

\theoremstyle{definition}
\newtheorem{defn}[prop]{Definition}
\newtheorem{empt}[prop]{}
\newtheorem{exm}[prop]{Example}
\newtheorem{rem}[prop]{Remark}

\newcommand{\vertiii}[1]{{\left\vert\kern-0.25ex\left\vert\kern-0.25ex\left\vert #1
		\right\vert\kern-0.25ex\right\vert\kern-0.25ex\right\vert}}
\newcommand{\Ga}{\Gamma}                     
\newcommand{\Coo}{C^\infty}                  

\usepackage[sc]{mathpazo}
\newbox\ncintdbox \newbox\ncinttbox 
\setbox0=\hbox{$-$} \setbox2=\hbox{$\displaystyle\int$}
\setbox\ncintdbox=\hbox{\rlap{\hbox
		to \wd2{\hskip-.125em \box2\relax\hfil}}\box0\kern.1em}
\setbox0=\hbox{$\vcenter{\hrule width 4pt}$}
\setbox2=\hbox{$\textstyle\int$} \setbox\ncinttbox=\hbox{\rlap{\hbox
		to \wd2{\hskip-.175em \box2\relax\hfil}}\box0\kern.1em}

\newcommand{\B}{\mathcal{B}}                 

\newcommand{\C}{\mathbb{C}}                  
\renewcommand{\H}{\mathcal{H}}               
\newcommand{\La}{\Lambda}                    

\newcommand{\N}{\mathbb{N}}                  
\newcommand{\eps}{\varepsilon}                    
\newcommand{\R}{\mathbb{R}}                  

\renewcommand{\th}{\theta}                   
\newcommand{\V}{\mathcal{V}}                 
\newcommand{\W}{\mathcal{W}}                 
\newcommand{\Z}{\mathbb{Z}}                  

\newcommand{\al}{\alpha}          
\newcommand{\bt}{\beta}           
\newcommand{\ga}{\gamma}          
\renewcommand{\th}{\theta}        

\def\<#1|#2>{\langle#1\stroke#2\rangle} 
\def\?#1|#2?{\{#1\stroke#2\}}        

\def\<#1,#2>{\langle#1,#2\rangle}            
\def\ee_#1{e_{{\scriptscriptstyle#1}}}       
\def\wick:#1:{\mathopen:#1\mathclose:}       

\newbox\ncintdbox \newbox\ncinttbox 
\setbox0=\hbox{$-$}
\setbox2=\hbox{$\displaystyle\int$}
\setbox\ncintdbox=\hbox{\rlap{\hbox
		to \wd2{\box2\relax\hfil}}\box0\kern.1em}
\setbox0=\hbox{$\vcenter{\hrule width 4pt}$}
\setbox2=\hbox{$\textstyle\int$}
\setbox\ncinttbox=\hbox{\rlap{\hbox
		to \wd2{\hskip-.05em\box2\relax\hfil}}\box0\kern.1em}

\newcommand{\stroke}{\mathbin|}   

\newcommand{\Hom}{\mathrm{Hom}}       
\newcommand{\Homeo}{\mathrm{Homeo}}       

\title{Cyclic noncommutative covering projections}
\begin{document}
\maketitle  \setlength{\parindent}{0pt}
\begin{center}
\author{
{\textbf{Petr R. Ivankov*}\\
e-mail: * monster.ivankov@gmail.com \\
}
}
\end{center}

\vspace{1 in}

\begin{abstract}
\noindent

\paragraph{}

The Gelfand - Na\u{i}mark theorem supplies the one to one correspondence between commutative $C^*$-algebras and locally compact Hausdorff spaces. So any noncommutative $C^*$-algebra can be regarded as a generalization of a topological space.  Generalizations of several topological invariants can be defined by algebraical methods. This article contains a pure algebraical construction of (noncommutative) covering projections with finite cyclic groups of covering transformations. 
\end{abstract}
\tableofcontents

\section{Motivation. Preliminaries}
\paragraph{}
Some notions of the geometry have noncommutative generalizations based on the Gelfand-Na\u{i}mark theorem.
\begin{thm}\label{gelfand-naimark}\cite{arveson:c_alg_invt} (Gelfand-Na\u{i}mark). 
	Let $A$ be a commutative $C^*$-algebra and let $\mathcal{X}$ be the spectrum of A. There is the natural $*$-isomorphism $\gamma:A \to C_0(\mathcal{X})$.
\end{thm}

\paragraph{}From the theorem it follows that a (noncommutative) $C^*$-algebra can be regarded as a generalized (noncommutative)  locally compact Hausdorff topological space. The following theorem gives a pure algebraic description of covering projections of compact spaces.
\begin{thm}\label{pavlov_troisky_thm}\cite{pavlov_troisky:cov}
	Suppose $\mathcal X$ and $\mathcal Y$ are compact Hausdorff connected spaces and $p :\mathcal  Y \to \mathcal X$
	is a continuous surjection. If $C(\mathcal Y )$ is a projective finitely generated Hilbert module over
	$C(\mathcal X)$ with respect to the action
	\begin{equation*}
	(f\xi)(y) = f(y)\xi(p(y)), ~ f \in  C(\mathcal Y ), ~ \xi \in  C(\mathcal X),
	\end{equation*}
	then $p$ is a finite-fold (or equivalently a finitely listed) covering.
\end{thm} 
\paragraph{} This article contains a pure algebraic construction  of finite-fold covering projections. 

\paragraph{} This article assumes an elementary knowledge of following subjects:
\begin{enumerate}
	\item Set theory \cite{halmos:set};
	\item Category theory  \cite{spanier:at};
	\item General topology \cite{munkres:topology};
	\item Algebraic topology \cite{spanier:at};

	\item $C^*$-algebras, $C^*$-Hilbert modules and $K$-theory  \cite{arveson:c_alg_invt, blackadar:ko,pedersen:ca_aut,murphy,takesaki:oa_ii}.
	
\end{enumerate}
\paragraph{}
The terms ``set'', ``family'' and ``collection'' are synonyms.

\paragraph{}Following table contains a list of special symbols.
\newline
\begin{tabular}{|c|c|}
	\hline
	Symbol & Meaning\\
	\hline
 & \\
	$A^G$  & Algebra of $G$-invariants, i.e. $A^G = \{a\in A \ | \ ga=a, \forall g\in G\}$\\
	$\mathrm{Aut}(A)$ & Group of *-automorphisms of $C^*$-algebra $A$\\
	$\B(\H)$ & Algebra of bounded operators on Hilbert space $\H$\\
	$\mathbb{C}$ (resp. $\mathbb{R}$)  & Field of complex (resp. real) numbers \\
	$C(\mathcal{X})$ & $C^*$-algebra of continuous complex valued \\
	& functions on a compact space $\mathcal{X}$\\
	$C_0(\mathcal{X})$ & $C^*$-algebra of continuous complex valued functions\\ 
	&  on a locally compact topological  space $\mathcal{X}$ equal to $0$ at infinity\\
	 $C_b(\mathcal{X})$ & $C^*$ - algebra of bounded  continuous complex valued \\
	  & functions on a topological space $\mathcal{X}$ \\
	$C_c(\mathcal{X})$ & Algebra of continuous complex valued functions on a \\
	 &  topological  space $\mathcal{X}$ with compact support\\
	$G(\widetilde{\mathcal{X}} | \mathcal{X})$ & Group of covering transformations of a covering projection  $\widetilde{\mathcal{X}} \to \mathcal{X}$ \cite{spanier:at}  \\
	$\H$ & Hilbert space \\
	$K_i(A)$ ($i = 0, 1$) & $K$ groups of $C^*$-algebra $A$\\
	$M(A)$  & The multiplier algebra of $C^*$-algebra $A$\\
	

	
	$U(A) \subset A $ & Group of unitary operators of algebra $A$\\
	
	$\mathbb{Z}$ & Ring of integers \\
	
	$\mathbb{Z}_n$ & Ring of integers modulo $n$ \\
	$\overline{k} \in \mathbb{Z}_n$ & An element in $\mathbb{Z}_n$ represented by $k \in \mathbb{Z}$  \\
	$X \backslash A$ & Difference of sets  $X \backslash A= \{x \in X \ | \ x\notin A\}$\\
	$|X|$ & Cardinal number of the finite set $X$\\ 
	$f|_{A'}$& Restriction of a map $f: A\to B$ to $A'\subset A$, i.e. $f|_{A'}: A' \to B$\\ 
	\hline
\end{tabular}
\paragraph{}


	



\subsection{Hilbert $C^*$-modules and compact operators}
\paragraph{} We refer to \cite{blackadar:ko} 
for definition of Hilbert $C^*$-modules, or simply Hilbert modules. If $X_A$ is a right Hilbert $C^*$-module then we denote by $\langle\cdot,\cdot\rangle_{X_A}$ the $A$-valued sesquilinear product on $X_A$, i.e. $\langle\xi,\eta\rangle_{X_A}\in A$; $\forall\xi,\eta \in X_A$.
 For any $\xi, \zeta \in X_A$ let us define an $A$-endomorphism $\theta_{\xi, \zeta}$ given by  $\theta_{\xi, \zeta}(\eta)=\xi \langle \zeta, \eta \rangle_{X_A}$ where $\eta \in X_A$. Operator  $\theta_{\xi, \zeta}$ is said to be a {\it rank one} operator and will be denoted by $\xi \rangle\langle \zeta$. The norm completion of an algebra generated by rank-one operators  $\theta_{\xi, \zeta}$ is said to be the {\it algebra of compact operators $\mathcal{K}(X_A)$}. We suppose that there is a left action of $\mathcal{K}(X_A)$ on $X_A$ which is $A$-linear, i.e. action of  $\mathcal{K}(X_A)$ commutes with action of $A$.

\subsection{Galois extensions and covering projections}\label{galois_subsection}

 \paragraph*{}This article compiles ideas of algebra and topology. 
 \begin{empt}\label{alg_top_constr}
 
 Following table contains the mapping between Galois extension of fields and topological covering projections.
\newline
\begin{tabular}{|c|c|}
	\hline
Topology & Theory of fields\\
\hline
	Covering projection & Algebraic extension of the field\\
	Regular covering projection & Normal extension\\
	Unramified covering projection  & 	Separable extension\\
	Universal covering projection  & 	Algebraic closure \\
	Covering projection with cyclic covering group  & 	Cyclic Galois extension \\
	\hline
\end{tabular}
 \paragraph*{}
 Complex algebraic varieties have structure of both functional fields \cite{hartshorne:ag} and Hausdorff locally compact spaces, i.e. both columns of the above table reflect the same phenomenon. Analogy between Galois extensions of fields and covering projections is discussed in \cite{hajac:toknotes, milne:etale}. During last decades the advanced theory of Galois extensions of noncommutative algebras had been developed.  Galois extensions became special cases of Hopf-Galois ones. There is the  Category Theory approach to Hopf-Galois extensions, based on theory of monads, comonads and adjoint functors \cite{hajac:toknotes}. But this approach cannot be used for nonunital algebras. The Category Theory replaces elements of algebras with endomorphisms of objects. However any object of the Category Theory contains the identity endomorphism which corresponds to the unity of the algebra. To avoid this obstacle the notion of the non-unital $C^*$-category was introduced \cite{mitchener:c_cat}. But this notion is incompatible with adjoint functors. So the Category Theory of Galois extensions cannot be directly used for nonunial $C^*$-algebras. But there is a modification of the Galois theory which  replaces the unity with the approximate unity.  Following theory which had been developed long time ago, is used as a prototype.
 
  \end{empt}
 \begin{empt}\textit{Galois theory of noncommutative algebras}. I follow to \cite{demeyer:genreal_galois,miyashita_fin_outer_gal}.
Let $K$ be a field. Throughout $\La$ will denote a $K$ algebra, $C$ will denote the center of $\La$ ($C=\mathfrak{Z}(\La)$). $G$ will denote a finite group represented as ring automorphisms of $\La$ and $\Ga$ the subring of all elements of $\La$ left invariant by all the automorphisms in $G$ ($\Ga = \La^G$). Let $\Delta\left(\La:G\right)$ be the crossed product of $\La$ and $G$ with trivial factor
set. That is 
$$\Delta\left(\La:G\right) = \sum_{\sigma \in G} \La U_\sigma \text{ such that}$$
$$
x_1U_\sigma x_2 U_\tau = x_1 \sigma(x_2)U_{\sigma  \tau}; ~ x_1, x_2 \in \La; ~ \sigma,\tau \in G.
$$
View $\La$ as a right $\Ga$ module and define $$j : \Delta\left(\La:G\right) \to \Hom_{\Ga}\left(\La, \La\right) \text{ by}$$ 
$$
j\left(aU_\sigma\right)\left( x\right)  = a \sigma(x);~a,x \in \La;~\sigma \in G. 
$$
 \end{empt}
 \begin{thm}\label{galois_equiv_thm}\cite{demeyer:genreal_galois} The following are equivalent:
 \begin{enumerate}
 \item[(a)]$\La$ is finitely generated projective as a right $\Ga$ module and
$j : \Delta\left(\La:G\right) \to \Hom_{\Ga}\left(\La, \La\right)$ is an isomorphism.
 \item[(b)] There exists $x_1,...,x_n, y_1,...,y_n \in \La$ such that
\begin{equation*}
\sum_{j =1}^{n} x_j\sigma(y_j)=\left\{
\begin{array}{c l}
   1 & \sigma \in  G \text{ is trivial}\\
   0 & \sigma \in  G \text{ is not trivial}
\end{array}\right.
\end{equation*}
 \end{enumerate}
 \end{thm}
\begin{defn}\label{galois_unital_defn}\cite{demeyer:genreal_galois}
If equivalent conditions of the Theorem \ref{galois_equiv_thm} hold then $\La$ is said to be a \textit{Galois extension} of $\Ga$.
\end{defn}
\section{Noncommutative finite covering projections}
\begin{defn}
	If $A$ is a $C^*$-algebra then an action of a group $G$ is said to be {\it involutive } if $ga^* = \left(ga\right)^*$ for any $a \in A$ and $g\in G$. Action is said to be \textit{non-degenerated} if for any nontrivial $g \in G$ there is $a \in A$ such that $ga\neq a$. 
\end{defn}
\paragraph{}A substantial feature of topological algebras is the existence of limits and infinite sums. We would like supply a $C^*$-algebraic  (nonunital) version  of the Theorem \ref{galois_equiv_thm} and the Definition \ref{galois_unital_defn} which used infinite sums instead finite ones. 
\begin{defn}\label{fin_def}
Let $\pi:A \to \widetilde{A}$ be an injective *-homomorphism of $C^*$-algebras, and let $G$ be a finite group such that following conditions hold:
\begin{enumerate}
\item[(a)] There is an  involutive non-degenerate  action of $G$ on $\widetilde{A}$ such that $\widetilde{A}^G=A$, where $\widetilde{A}^G$ is the algebra of $G$-invariants, i. e. $\widetilde{A}^G = \left\{\widetilde{a} \in \widetilde{A}~|~g\widetilde{a}=\widetilde{a};~\forall g\in G\right\}$; 
\item[(b)] There is a finite or countable set $I$ and an indexed by $I$ subset $\{a_{\iota}\}_{\iota \in I} \subset \widetilde{A}$ such that
\begin{equation}\label{can_nc_eqn}
\sum_{\iota \in I} a_{\iota}(ga^*_\iota)=\left\{
\begin{array}{c l}
    1_{M(\widetilde{A})} & g \in  G \text{ is trivial}\\
   0 & g \in  G \text{ is not trivial}
\end{array}\right.
\end{equation}
where the sum of the series means the strict convergence \cite{blackadar:ko}.

\end{enumerate}

Then $\pi$ is said to be a {\it finite  noncommutative covering projection}. $G$ is said to be the {\it covering transformation group}. Denote by $G\left(\widetilde{A}~|~A\right)=G$. The algebra $\widetilde{A}$ is said to be the {
\it covering algebra}, and $A$ is called the {\it base algebra} of the covering projection. A triple $\left(A, \widetilde{A}, G\right)$ is also  said to be a {\it finite  noncommutative covering projection}.
\end{defn}
\begin{rem}
From the definition \ref{fin_def} it follows that both $A$ and $\widetilde{A}$ are $\sigma$-unital.
\end{rem}
\begin{rem}
	The Theorem \ref{pavlov_troisky_thm} gives an algebraic characterization of finitely listed covering projections of compact Hausdorff topological spaces and covering projections should not be regular, the Definition \ref{fin_def} concerns with regular covering projections of locally compact spaces.
\end{rem}
\begin{rem}
	The Definition  \ref{fin_def} is motivated by the Theorem \ref{comm_fin_thm}.
\end{rem}

\begin{defn}\label{hilbert_product_defn}
	Let $\left(A, \widetilde{A}, G\right)$ be a  finite  noncommutative covering projection.  Algebra  $\widetilde{A}$  is a  countably generated  Hilbert $A$-module with a sesquilinear product given by
	\begin{equation}\label{fin_form_a}
	\left\langle a, b \right\rangle_{\widetilde{A}} = 
	 \sum_{g \in G} g(a^*b).
	\end{equation}
	We say that the structure of Hilbert $A$-module is {\it induced by the covering projection} $\left(A, \widetilde{A}, G\right)$. Henceforth we shall consider $\widetilde{A}$ as a right $A$-module, so we will write $\widetilde{A}_A$. 
\end{defn}

\begin{rem}
	In the commutative case  the product \eqref{fin_form_a}  coincides with the product given by the Theorem \ref{pavlov_troisky_thm}.
\end{rem}
\begin{rem}
From \eqref{can_nc_eqn} it follows that
\begin{equation}\label{can_nc_hilb_eqn}
1_{M\left(A\right)}= \sum_{\iota \in I} \left. a_\iota \right\rangle \left\langle a_\iota \right.
\end{equation}
where above series is strictly convergent.
 \end{rem}

 \paragraph{}   The    algebra of compact operators $\mathcal{K}\left(\widetilde{A}_A\right)$ as well as $\widetilde{A}$ non-degenerately acts on $\widetilde{A}_A$. So we can formulate the following lemma.
    \begin{lem}\label{fin_compact_lem}
    If  $\left(A, \widetilde{A}, G\right)$ is a finite noncommutative covering projection then  $\widetilde{A} \subset\mathcal{K}\left(\widetilde{A}_A\right)$.
    \end{lem}
\begin{proof}
If $\widetilde{a}\in \widetilde{A}$ then from \eqref{can_nc_hilb_eqn} it follows that
$$
\widetilde{a} = \sum_{\iota \in I} \left. a_\iota \right\rangle \left\langle a_\iota\cdot \widetilde{a}\right.  =\sum_{\iota \in I} \left. a_\iota \right\rangle\left\langle a_\iota  \widetilde{a} \right..
$$
The series $\sum_{\iota \in I} \left. a_\iota \right\rangle\left\langle a_\iota  \widetilde{a} \right.$ is norm convergent because $\sum_{\iota \in I} \left. a_\iota \right\rangle\left\langle a_\iota\right.$ is strictly convergent. So $\widetilde{a}$ is compact, i.e. $\widetilde{a} \in \mathcal{K}\left(\widetilde{A}_A\right)$.
\end{proof}
\begin{rem}
The Lemma \ref{fin_compact_lem} may be regarded as a generalization of the Theorem \ref{galois_equiv_thm} because if $\widetilde{A}$ is unital then $\widetilde{A}$ is finitely generated projective right $A$-module, if and only if  $\widetilde{A} \subset\mathcal{K}\left(\widetilde{A}_A\right)$ (See \cite{clare_crisp_higson:adj_hilb}).
\end{rem}
\begin{exm}\label{circle_fin}{\it Finite covering projections of the circle $S^1$} 
Let us define the angular parameter $\varphi \in \R$ on $S^1$ such that circle $S^1 \subset \R^2$ and if $\left(x, y\right) \in S^1 \subset \R^2$ then
$$
x = \cos\left(\varphi\right), ~~ y = \sin\left(\varphi\right).
$$
Denote by $\tau_{s} \in \Homeo\left(S^1 \right)$ given by 
\begin{equation}\label{tau_aut_eqn}
\tau_s\left( \varphi\right)  = \varphi + s
\end{equation}

and denote by $\tau_{s}: C\left(S^1\right) \approx C\left(S^1 \right)$ corresponding *-automorphism.  

There is the universal covering projection $\widetilde{\pi}: \mathbb{R}\to S^1$. 
Let $\widetilde{\mathcal{U}}_1,\ \widetilde{\mathcal{U}}_2 \subset \mathbb{R}$ be such that
\begin{equation*}
\widetilde{\mathcal{U}}_1 = (-\pi -1/2,  1/2), \ \widetilde{\mathcal{U}}_2 = (-1/2, \pi +  1/2). 
\end{equation*} For any $i \in \{1,2\}$ the set $\mathcal{U}_i = \widetilde{\pi}(\widetilde{\mathcal{U}}_i) \subset S^1$ is open, connected and evenly covered. Since $S^1=\mathcal{U}_1\bigcup\mathcal{U}_2$ there is  a partition of unity $a_1, a_2$ dominated by $\{\mathcal{U}_j\}_{j\in \{1,2\}}$ \cite{munkres:topology}, i.e.  $a_j: S^1 \to [0,1]$ are such that
\begin{equation*}
\left\{
\begin{array}{c l}
    a_j(x)>0 & x \in \mathcal{U}_j \\
    a_j(x)=0 & x \notin \mathcal{U}_j
\end{array}\right.; \ j \in \{1,2\}.
\end{equation*}
and $a_1+a_2 = 1_{C(S^1)}$. We suppose that $a_1, a_2 \in \Coo\left(S^1\right)$ for the following proofs. 
If $e_1, e_2\in C\left(S^1\right)$ are given by
\begin{equation*}
e_j = \sqrt{a_j}; \ j=1,2;
\end{equation*}
then $e_j \in \Coo\left( S^1\right)$ and
\begin{equation*}
\left(e_1\right)^2+\left(e_2\right)^2= 1_{C_0(S^1)};~ e^*_1 = e_1;~e^*_2=e_2.
\end{equation*}
Let $\widetilde{e}_j \in C_c(\mathbb{R})$ be given by
\begin{equation}\label{tilde_e_eqn}
\widetilde{e}_j(\widetilde{x})=\left\{
\begin{array}{c l}
e_j(\widetilde{\pi}(\widetilde{x})) & \widetilde{x} \in \widetilde{\mathcal{U}}_i \\
0 & \widetilde{x} \notin \widetilde{\mathcal{U}}_j
\end{array}\right.; \ j \in \{1,2\}.
\end{equation}
The *-homomorphism
$$
C(\pi_n) : C\left(u \right) \to C\left(v \right) 
$$
$$
u \mapsto v^n
$$
corresponds to a $n$-listed covering projection $\widetilde{\pi}^n:S^1_v \to S^1_u$, where $S^1_u \approx S^1_v \approx S^1$.
It is well known that $G\left( S^1_v~|~S^1_u\right) \approx \mathbb{Z}_n$.  From $C\left(S^1_u \right)= C\left(S^1_v \right)^{\mathbb{Z}_n} $  it follows that the condition (a) of the Definition \ref{fin_def} holds. There is a sequence of covering projections $\mathbb{R}\xrightarrow{\pi^n}S^1_v \to S^1_n$. If $\mathcal{U}^n_i = \pi^n\left(\widetilde{\mathcal{U}}_i\right)$ then $\mathcal{U}^n_i \bigcap g\mathcal{U}^n_i = \emptyset$ for any nontrivial $g \in G(S^1_v~|~ S^1_u)$. If $e^n_i \in C(S^1_v)$ is given by
\begin{equation}\label{e_n_i}
e^n_j\left(x\right)   =\left\{\begin{array}{c l}
    \widetilde{e_j}\left(\widetilde{x} \right) &  x \in \mathcal{U}^n_j ~\&~ x= \widetilde{\pi}^n\left(\widetilde{x} \right) ~\&~ \widetilde{x} \in \widetilde{\mathcal{U}}_j  \\
   0 & x \notin \mathcal{U}^n_j
\end{array}\right.
\end{equation}
then
\begin{equation*}
\sum_{j \in \{1,2\}; \ g \in G\left(S^1_v~|~S^1_u\right)} g\left(e^n_j\right)^2 = 1_{C_0\left( S^1_v\right) }; 
\end{equation*}
and from $g \mathcal U^n_j \bigcap \mathcal U^n_j=\emptyset$ for any notrivial $g \in G\left(S^1_v~|~S^1_u\right)$ it follows that
\begin{equation*} 
 e^n_j(ge^n_j) = 0; \ \text{ for any notrivial } g \in  G\left(S^1_v~|~S^1_u\right).
\end{equation*}
If $I_n = G\left(S^1_v~|~S^1_u\right)\times \{1,2\}$
and
\begin{equation}\label{iota_def_eqn}
e_{\iota} = ge_j^n; \text{ where } \iota = (g,j) \in I_n
\end{equation}
then
\begin{equation}\label{circ_sum_eqn}
\sum_{\iota \in I_n} e_{\iota}(ge^*_\iota)=\sum_{\iota \in I_n} e_{\iota}(ge_\iota)=\left\{
\begin{array}{c l}
    1_{C\left(S^1_v \right) } & g \in  G\left(S^1_v~|~S^1_u\right) \text{ is trivial}\\
   0 & g \in  G\left(S^1_v~|~S^1_u\right) \text{ is not trivial}
\end{array}\right..
\end{equation}

So a natural *-homomorphism $\pi:C\left(S^1_u\right)\to C\left(S^1_v\right)$ satisfies the condition (b) of the Definition \ref{fin_def}.  So a triple $\left(C(S^1_u),C_0(S^1_v), \mathbb{Z}_n\right)$ is a finite noncommutative covering projection.
\end{exm}
\begin{empt}
The $n$-listed covering projection of the circle corresponds to $*$-homomorphism
$$
C\left(u \right) \to C\left(v\right)  
$$
$$
u \mapsto v^n 
$$
and the action of $\Z_n$ on $C\left(u \right)$ is given by
$$
\overline{k}v = e^{\frac{2\pi i k}{n}}v.
$$ 
Clearly

$$
\left\langle v^j, v^k\right\rangle_{C\left( v\right) } = \sum_{\overline{l} \in \Z_n} \overline{l}\cdot v^{j - k} = \sum_{l = 0}^{n -1} e^{\frac{2\pi i (j-k)l}{n}}v^{j-k} = \left\{ \begin{array}{c l}
    n & k \equiv j \text{ mod } n\\
     0 & k \not \equiv j \text{ mod } n
\end{array}\right..
$$

From above equation it follows that
\begin{equation}\label{circ_frame_eqn}
1_{C\left(v \right) } = \frac{1}{n}\sum_{j = 0}^{n-1}\left. v^j \right\rangle \left\langle v^j \right..
\end{equation}

\end{empt}
\begin{thm}\label{comm_fin_thm}\cite{ivankov:inv_nc_cov_limits}
Following conditions hold:
\begin{enumerate}
\item[(a)] If $\widetilde{p}:\widetilde{\mathcal X} \to \mathcal X$ is a finite  topological regular covering projection and both $\mathcal X$ and $\widetilde{\mathcal X}$ are locally compact Hausdorff spaces then $\left(C_0\left(\mathcal X \right)  , C_0\left(\widetilde{\mathcal X} \right), G\left(\widetilde{\mathcal X} ~|~ \mathcal X \right) \right)$ is a finite noncommutative covering projection;
\item[(b)] If $\mathcal X$ and  $\widetilde{\mathcal X}$ are locally compact Hausdorff spaces and $\left(C_0\left(\mathcal X \right)  , C_0\left(\widetilde{\mathcal X} \right), G \right)$ is a finite noncommutative covering projection then the natural continuous map   $\widetilde{p}:\widetilde{\mathcal X} \to \mathcal X$ is a  regular finitely listed covering projection such that $G\left(\widetilde{\mathcal X} ~|~ \mathcal X \right) \approx G$.  
\end{enumerate}
\end{thm}


\section{Noncommutative \'etale fiber products}

\begin{defn}\label{fiber_prod_defn}
	Let $A$, $\widetilde{A}$, $B$, $\widetilde{B}$ be $C^*$-algebras such that following conditions hold:
\begin{enumerate}
	\item[(a)]  $A \subset \widetilde{A}$ and $\left(A, \widetilde{A}, G\right)$ is a finite noncommutative covering projection,
	\item[(b)]  $B$ is $\sigma$-unital,
	\item[(c)]There are  inclusions of $C^*$-algebras $A \subset M\left(B \right)$, $B \subset \widetilde{B}$,
	\item[(d)] If $\widetilde{A} \otimes_A B$ is the algebraic tensor product then there is the injective $\widetilde{A}$-$B$ bimodule homomorphism  $$\varphi: \widetilde{A} \otimes_A B \to\widetilde{B} \text{ such that }
	\varphi\left(\widetilde{a} \otimes b \right) = \widetilde{a}b
	$$
	and $\widetilde{B}$ is the $C^*$-norm completion of $\varphi\left( \widetilde{A} \otimes_A B \right)$. 
	\item[(e)] The natural action of $G$ on  $\widetilde{A} \otimes_A B$ induces the involutive action of $G$ on $\widetilde{B} $, i.e. there is the action  $G$ on $\widetilde{B}$ such that
	$$
	\varphi(g \widetilde{a}\otimes b) = g\left(\widetilde{a} b \right)  ; \ \forall g \in G.
	$$

	\end{enumerate}
	The quadruple $\left(A, \widetilde{A}, B, \widetilde{B}  \right)$ is said to be a  \textit{(noncommutative) \'etale fiber product}.
	\end{defn}

\begin{lem}
If	 $\left(A, \widetilde{A}, B, \widetilde{B}  \right)$ is a noncommutative \'etale fiber product then the triple $\left( B, \widetilde{B}, G\right)$ is a finite  noncommutative covering projection.
 
\end{lem}
\begin{proof}
	Since $B$ is $\sigma$-untal there is a finite or countable set of self-adjoint elements $\left\{b_{\iota'}\in B\right\}_{\iota' \in I'}$ such that
	\begin{equation}\label{sum_b_eqn}
	\sum_{\iota' \in I'}b^2_{\iota'} = 1_{M\left(B \right) }.
	\end{equation}
	Otherwise there is a finite or countable subset $\left\{\widetilde{a}_{\iota''}\in A\right\}_{\iota'' \in I''}$ such that 
	\begin{equation}\label{a_iota''eqn}
\sum_{\iota'' \in I''} \widetilde{a}_{\iota''}(g\widetilde{a}^*_{\iota''})=\left\{
\begin{array}{c l}
1_{M(\widetilde{A})} & g \in  G \text{ is trivial}\\
0 & g \in  G \text{ is not trivial}
\end{array}\right.	
\end{equation}

Let us consider a finite or countable set $I = I' \times I''$ and the subset $\left\{\widetilde{b}_\iota \right\}_{\iota \in I} \subset \widetilde{B}$ such that
$$
\widetilde{b}_\iota = \widetilde{a}_{\iota''} b_{\iota'}; ~~\iota = \left(\iota', \iota'' \right)  \in I.
 $$
 Clearly $\widetilde{b}_\iota^*=\widetilde{a}^*_{\iota''}b_{\iota'}$. From \eqref{sum_b_eqn} and \eqref{a_iota''eqn} it follows that
 $$
\sum_{\iota \in I} \widetilde{b}_{\iota}(g\widetilde{b}^*_{\iota})=\sum_{\left( \iota', \iota'\right)  \in I' \times I''}\widetilde{a}_{\iota''}b_{\iota'}\left(g\left(  b_{\iota'}\widetilde{a}^*_{\iota''}\right) \right) =
$$
$$
= \sum_{\iota'' \in I''} \widetilde{a}_{\iota''}\left(\sum_{\iota' \in I'} b_{\iota'}^2 \right)g\widetilde{a}^*_{\iota''} = \sum_{\iota'' \in I''} \widetilde{a}_{\iota''}(g\widetilde{a}^*_{\iota''}) \left\{
\begin{array}{c l}
1_{M(\widetilde{B})} & g \in  G \text{ is trivial}\\
0 & g \in  G \text{ is not trivial}
\end{array} \right. .
 $$
\end{proof}

\begin{defn}\label{org_tensor_prod_defn}\cite{lang}
Let $R$ be a commutative ring and let both $A$,  $B$ are $R$-algebras, such that the image of $R$ is contained in centers of both $A$ and $B$. We shall make $A\otimes_R B$ into $R$-algebra. Given $\left( a, b\right) \in A \times B$, we have an $R$-bilinear map 
$$
M_{a,b}: A \times B \to A \otimes_R B \text{ such that }M_{a,b}\left(a',b' \right)= aa'\otimes bb'. 
$$
Hence $M_{a,b}$ induces an $R$-linear map $m_{a,b} A\otimes_R B \to A\otimes_R B$ such that $m_{a,b}\left(a',b' \right)= aa'\otimes bb'$. But $m_{a,b}$ depends bilinearly on $a$ and $b$, so we obtain finally a unique $R$-bilinear map
$$
\left( A\otimes_R B\right) \times \left(A\otimes_R B \right) \to A\otimes_R B
$$
such that $\left(a \otimes b \right) \left(a \otimes b \right) = aa' \otimes b b'$. This map is obviously associative, and we have a natural ring homomorphism
$$
R \to A \otimes_R B \text{ given by } c \mapsto 1 \otimes c = c \otimes 1.
$$
Thus $A \otimes_R B$ is an $R$ algebra, called the \textit{ ordinary tensor product}. 

\end{defn}

\begin{exm} \textit{Fiber products of locally compact spaces}.
Let $\mathcal{X}$ be a locally compact second countable Hausdorff space, and let $\pi_{\widetilde{\mathcal{X}}}:\widetilde{\mathcal{X}} \to \mathcal{X}$ be a finite regular covering projection. Let $\mathcal{Y}$ be a locally compact second countable Hausdorff space, and let $\pi_{\mathcal{Y}}:\mathcal{Y} \to \mathcal{X}$ be a surjective continuous map. Suppose that $ \mathcal{X}, \widetilde{\mathcal{X}}, \mathcal{Y}$ are connected spaces. Algebras $C_0\left( \mathcal{X}\right)$ 
,  $C_0\left( \widetilde{\mathcal{X}}\right)$ ,  $C_0\left(\mathcal{Y}\right)$ are $\sigma$-unital.
Let us consider a fiber product
$
\widetilde{\mathcal{Y}} = \widetilde{\mathcal{X}} \times_{\mathcal{X}} \mathcal{Y}$ of topological
given by the following diagram.
		\newline
	\begin{tikzpicture}
	\matrix (m) [matrix of math nodes,row sep=3em,column sep=4em,minimum width=2em]
	{
\widetilde{\mathcal{Y}}	& \mathcal{Y}   \\
\widetilde{\mathcal{X}} & \mathcal{X} \\};
	\path[-stealth]
	(m-1-1) edge node [left] {} (m-1-2)
	(m-1-2) edge node [right] {} (m-2-2)
	(m-1-1) edge node [right] {} (m-2-1)
	(m-2-1) edge node [above] {}  (m-2-2);
	
	\end{tikzpicture}
		\newline
		Let us prove that the  $\left(C_0\left(\mathcal X \right) , C_0\left(\widetilde{\mathcal X} \right), C_0\left(\mathcal Y \right) , C_0\left(\widetilde{\mathcal Y} \right)  \right)$ is noncommutative \'etale fiber product. If $G = G\left( \widetilde{\mathcal{X}}~|~\mathcal{X}\right)$ is covering transformation group then   from the Theorem \ref{comm_fin_thm} it follows that the triple $\left(C_0\left( \mathcal{X}\right) ,C_0\left(  \widetilde{\mathcal{X}}\right) ,G\right)$ is a finite  noncommutative covering projection, i.e.  the condition (a) of the Definition \ref{fiber_prod_defn} holds. Since $\mathcal{Y}$ is second-countable the $C^*$-algebra $C_0\left(\mathcal Y \right)$ is $\sigma$-unital, i.e. condition (b) holds. 
	Any point $\widetilde{y} \in \widetilde{\mathcal{X}} \times_{\mathcal{X}} \mathcal{Y}$ can be represented by the pair $\left(\widetilde{x}, y \right)  \in \widetilde{\mathcal{X}} \times \mathcal{Y}$ such that $\pi_{\widetilde{\mathcal{X}}}\left( \widetilde{x}\right) = \pi_{\mathcal{Y}}\left(y \right)$.	
	There is the natural action of $G$ on the fiber product
		$$
		G \times \left(\widetilde{\mathcal{X}} \times_{\mathcal{X}} \mathcal{Y} \right) \to \widetilde{\mathcal{X}} \times_{\mathcal{X}} \mathcal{Y}
		$$
		$$
\left( 	g,	\left(\widetilde{x}, y \right)\right) \mapsto \left(g\widetilde{x}, y \right).
		$$

	Let $\widetilde{y} \in \widetilde{\mathcal{Y}}$ is represented by the pair $\left(\widetilde{x}, y \right)  \in \widetilde{\mathcal{X}} \times \mathcal{Y}$ and let $x = \pi_{\widetilde{\mathcal{X}}}\left(\widetilde{x} \right)$. If $\mathcal U$ is an open connected neighborhood of $x$ evenly covered by $\pi_{\widetilde{\mathcal{X}}}$ then the set $\mathcal V= \pi^{-1}_{\mathcal{Y}}\left( \mathcal{U}\right)  $ is an open evenly covered by $\pi: \widetilde{\mathcal{Y}} \to \mathcal Y $ is a regular covering projection. So the triple  $\left(C_0\left( \mathcal{Y}\right) ,C_0\left(  \widetilde{\mathcal{Y}}\right) ,G\right)$ is a finite  noncommutative covering projection and there is the natural injective *-homomorphism $C_0\left( \mathcal{Y}\right) \to C_0\left(\widetilde{ \mathcal{Y}}\right)$. Otherwise  $\mathcal Y \to \mathcal X$ is surjective, hence there is the natural injective homomorphism $C_0\left(\mathcal{X} \right) \to C_b\left(\mathcal{Y} \right)= M\left( C_0\left(\mathcal{Y} \right)\right) $, i.e. condition (c) holds.	The projection $\widetilde{\mathcal{X}} \times \mathcal{Y}\to \widetilde{\mathcal{X}}$ induces the surjective continuous maps:
				$$
				\pi_{\widetilde{\mathcal{Y}}}: \widetilde{\mathcal{Y}} \to \widetilde{\mathcal{X}}, ~
				\pi: \widetilde{\mathcal{Y}} \to \mathcal Y.
				$$
				It follows that there are natural injective *-homomorphisms $C_0\left(\widetilde{\mathcal X}\right)  \to C_b\left(\widetilde{\mathcal Y}\right) $,  $C_0\left(\mathcal Y\right)  \to C_0\left(\widetilde{\mathcal Y}\right) $
				According to the Definition \ref{org_tensor_prod_defn} the tensor product $C_0\left(  \widetilde{\mathcal{X}}\right) \otimes_{C_0\left( \mathcal{X}\right)} C_0\left(  \mathcal{Y}\right)$ has the structure of the commutative $C_0\left( \mathcal{X}\right)$-algebra. Moreover there is the natural injective $\varphi:C_0\left(  \widetilde{\mathcal{X}}\right) -C_0\left(  \mathcal{Y}\right)$ bimodule homomorphism 
	$$\varphi:C_0\left(  \widetilde{\mathcal{X}}\right) \otimes_{C_0\left( \mathcal{X}\right)} C_0\left(  \mathcal{Y}\right) \to C_0\left(  \widetilde{\mathcal{Y}}\right)
	$$ 
	of 	$C_0\left( \mathcal{X}\right)$ algebras such that
	$$
	\varphi\left(a \otimes \widetilde{b} \right) = a \widetilde{b}.
	$$
	 Let $\left\langle \cdot, \cdot\right\rangle_{C_0\left(  \widetilde{\mathcal{Y}}\right) }$ be the $ C_0\left(  \mathcal{Y}\right)$-valued product on $C_0\left(  \widetilde{\mathcal{Y}}\right)$  induced by the covering projection $\left(C_0\left( \mathcal{Y}\right) ,C_0\left(  \widetilde{\mathcal{Y}}\right) ,G\right)$. The product can be uniquely extended to $C_b\left(  \widetilde{\mathcal{Y}}\right)$.	
	If $\left\{\widetilde{a}_\iota \right\}_{\iota \in I} \subset C_0\left(\widetilde{\mathcal{Y}} \right)$ is a finite or countable set such that
$$
 \sum_{\iota \in I} \widetilde{a}_{\iota}\left( g \widetilde{a}_\iota\right)=
\left\{\begin{array}{c l}
1_{C_b\left( \widetilde{\mathcal{X}}\right) } = 1_{C_b\left( \widetilde{\mathcal{Y}}\right) } & g \in  G \text{ is trivial}\\
0 & g \in  G \text{ is not trivial}
\end{array} \right.
$$ 
then
$$
1_{C_b\left( \widetilde{\mathcal{Y}}\right) } = \sum_{\iota \in I} a_\iota \left\rangle \right\langle a_\iota.
$$
It follows that for any $\widetilde{b}\in C_0\left( \widetilde{\mathcal{Y}}\right)$ following condition holds
$$
\widetilde{b} = \sum_{\iota \in I} \widetilde{a}_\iota \left\langle \widetilde{b}, \widetilde{a}_\iota\right\rangle = \sum_{\iota \in I} \widetilde{a}_\iota b_\iota, \text{ where } ~ b_\iota \in C_0\left(\mathcal Y \right)
$$
and  above series is $C^*$-norm convergent. Otherwise for any finite $I_0 \in I$ following condition holds
$$
\sum_{\iota \in I_0} \varphi\left(  \widetilde{a}_\iota \otimes b_\iota\right) = \sum_{\iota \in I_0} \widetilde{a}_\iota b_\iota
$$
It follows that the set $ \varphi:C_0\left(  \widetilde{\mathcal{X}}\right) \otimes_{C_0\left( \mathcal{X}\right)} C_0\left(  \mathcal{Y}\right)$  is dense in $C_0\left(  \widetilde{\mathcal{Y}}\right)$ with respect to $C^*$-norm, i,e. the condition (d) holds. The action $G$ on $C_0\left(\widetilde{\mathcal Y} \right)$ is such that $C_0\left(\widetilde{\mathcal Y} \right)^G = C_0\left(\mathcal Y \right)$, hence
$$
	\varphi(g \widetilde{a}\otimes b) = g\left(\widetilde{a} b \right) \text{ where } \widetilde{a} \in C_0\left(\widetilde{\mathcal X} \right), ~ b \in C_0\left(\mathcal Y \right),~ g \in G,
$$
i.e. condition (e) holds.

\end{exm}

\section{Cyclic covering projections}\label{cov_pr_and_k_sec}

\paragraph{} A cyclic extension of a field is a Galois extension with a finite Galois group \cite{weil:basic_number_theory}. Discussed here cyclic covering projections are analogs of cyclic fields' extensions.

\begin{defn}\label{cyclic_projection_defn}
	Let $A$ be a $C^*$-algebra, and let $A\rightarrow \B(\H)$ be a faithful representation. Let $u \in U\left( M\left( A\right) \right) $ be an unitary element.
	 Let $v\in \B(\H)$ be such that $v^n=u$ and $v^j \notin U(M(A))$ for any $j=1,..., n-1$. 
	The the $C^*$-norm completion of the algebra generated by operators 
	\begin{equation*}
v^jav^k \text{ where } a \in A; \ j,k \in \{0, ..., n-1\};
	\end{equation*}
	is said to be  the {\it cyclic extension} (\textit{generated by} $v$).
	Denote by $A[v]$ the generated by $v$ cyclic extension. It is clear that $v$ is a multiplier of $A[v]$.
	Number $n$ is said to be the {\it degree} of the extension.
\end{defn}
\begin{lem}\label{ext_gal_lem}
If $u \in U\left( M\left(A \right) \right) $ and $A[v]$ is a cyclic extension of degree $n$ then the quadruple $\left(C\left(u \right) , C\left(v \right) ,A, A[v]  \right)$ is  a  noncommutative \'etale fiber product.
\end{lem}
\begin{proof}
Any $C^*$-algebra can be defined by its irreducible representations. Let $\rho : A[v] \to \B\left( \H_\rho\right) $ be an irreducible representation. The representation induces the irreducible representation of  $C\left( v\right) $ and there is $z_\rho \in \C$ such that

$$
\rho\left( v\right) \xi = z_\rho\xi; ~~ \forall \xi \in \H_\rho, ~~ \left|z_\rho \right|= 1.
$$
For any $\overline{k} \in \Z_n$ we define an action $\rho_{\overline{k}}: A[v] \to \B\left(\H_\rho \right) $ given by

$$
\rho_{\overline{k}}\left(a\right) \xi =\rho\left(a\right) \xi; ~~ \forall a \in A, ~~ \forall \xi \in \H_\rho,
$$
$$
\rho_{\overline{k}}\left(v\right)  \xi = z_\rho e^{\frac{2 \pi i k}{n}} \xi. 
$$
These actions define an involutive action 
$\Z_n \times A[v] \to A[v]$ such that for any $\overline{k} \in \Z_n$ following conditions hold:

$$
\overline{k} a = a; ~ \forall a \in A,
$$
$$
\overline{k} v = v e^{\frac{2 \pi i k}{n}}.
$$

 $A[v]$ is an $A$-Hilbert module with the following scalar product

$$
	\left\langle a, b \right\rangle_{A[v]} = 
	 \sum_{g \in G} g(a^*b).
$$

From \eqref{circ_frame_eqn} it follows that

$$
1_{A[v] } = \frac{1}{n}\sum_{j = 0}^{n-1}v^j\rangle\langle v^j
$$
hence any $\widetilde{a} \in A[v]$ can be represented by following way
$$
\widetilde{a} = \frac{1}{n} \sum_{j = 0}^{n - 1} v^j \left\langle v^j, \widetilde{a}\right\rangle = \sum_{j = 0}^{n - 1} v^ja_j; ~~ \text{ where } a_j \in A.
$$

It follows that
$$
A[v] = C\left( v\right) \otimes_{C\left(u \right) } A.
$$

\end{proof}

\begin{rem}
	From the Lemma \ref{ext_gal_lem} it follows that
	\begin{equation}\label{hilb_compat_1_eqn}
	1_{A[v]}= \sum_{\iota \in I}\left. e_\iota\left( v\right)  \right\rangle \left\langle e_\iota\left(v \right)\right..
	\end{equation}
	
\end{rem}


\begin{defn}\label{root_n_defn} Denote $\mathbb{C}^*=\{z \in \mathbb{C} \ | \ |z| = 1 \}$.
	A {\it $n^{\mathrm{th}}$ root of the identity map} is a Borel-measurable function $\mu \in \ \B_{\infty}(\mathbb{C}^*)$  such that
	\begin{equation*}
	\begin{split}
	\mu^n = \mathrm{Id}_{\mathbb{C}^*}; \\ \mu^i \neq \mathrm{Id}_{\mathbb{C}^*}; \ 0 < i < n.
	\end{split}
	\end{equation*}
		Let $\varphi$ be the angular parametrization of $S^1$ such that $ 0 \le \varphi < 2 \pi$. The \text{ standard $n^{\text{th}}$ root of identity map} is given by the Borel-measurable function
		$$
		\varphi \mapsto e^{\frac{i\varphi}{n}}. 
		$$
		Denote by $\mu_{n}$ the standard $n^{\text{th}}$ root of identity map.
\end{defn}

\begin{rem}
Let $A \to  B\left(\H \right)$ be a faithful representation of $C^*$-algebra $A$.  
	Let $u \in U\left(M(A)\right)$ be
		 such that there is no element $w \in U\left( M\left( A\right) \right) $ such that  $w^n=u$ and $w^i \notin U(M(A))$ for any $i\in \left\{1,..., n-1\right\}$.  If $\mu$ is an $n^{\mathrm{th}}$ root of the identity map then $v = \mu(u)\in \B(\H)$ is such that $v^n=u$ and $v^i \notin U\left(M(A)\right)$ for any $i=1,..., n-1$.  
\end{rem}
\begin{defn}
	A {\it $n^{\mathrm{th}}$ root of the unity} is a Borel-measurable function $\nu \in \ B_{\infty}(\mathbb{C}^*)$  such that
	\begin{equation*}
	\begin{split}
	\nu^n = 1 \in \C.
	\end{split}
	\end{equation*}

\end{defn}

\begin{lem}\label{unity_identity_lem}
Any $n^{\mathrm{th}}$ root $\mu$ of the identity map can be uniquely represented as following product
$$
\mu = \mu_n \nu
$$ 
where $\mu_n$ is a standard  $n^{\text{th}}$ root of identity map and $\nu$ is a  root of the unity. If automorphism $\tau_s$ is given by \eqref{tau_aut_eqn} then 
$$
\tau_s\left( \mu_n\right) = e^{2\pi i s} \mu_n \nu'
$$

where $\nu'$ a  root of the unity.
\end{lem}
\begin{proof}
Evident.
\end{proof}
\begin{cor}
If $\mu$ is any  $n^{\text{th}}$ root of identity map then for any $s \in \R$ there is an $n^{\text{{th}}}$ root of unity such that
$$
\tau_s\left( \mu\right) = e^{2\pi i s} \mu \nu.
$$
\end{cor}

\begin{lem}\label{homotopy_u_lem}
		If $u_t :\left[0, 1\right] \to U\left(M\left( A\right)  \right)$ is a homotopy and $\mu$ be an $n^{\text{th}}$ root of identity map then there is $w \in  U\left(M\left( A\right)\right) $ such that
		\begin{equation*}
		\mu\left(u_1 \right) = w \mu\left(u_0 \right).
		\end{equation*}
	\end{lem}
	\begin{proof}
			There is an open subset $\V \in U\left(M\left( A\right)  \right)$ such that $1_A \in \V$ and the restriction $f|_{\V}$ of the map $f = u \mapsto u^n$ is a homeomorphism onto the open set $\mathcal{W} = f\left(\V \right)  \subset U\left(M\left( A\right) \right)$. Clearly  $\W \subset  U\left(M\left( A\right) \right)$ is  open neighborhood of $1_A$. There is the bijective inverse map $f^{-1} : \W \to \V$.
		For any $t \in \left[0, 1\right]$ there is $\eps > 0$, such that
		$$
		\left|t' - t\right| < \eps \Rightarrow u_{t'} u^{-1}_{t} \in \W.
		$$
		It follows that
		$$
		\mu\left(u_{t'} u^{-1}_{t} \right) = f^{-1}\left(u_{t'} u^{-1}_{t} \right) \in U\left(M\left( A\right) \right) 
		$$
		$$
		\mu\left(u_{t'} \right) = w \mu\left( u_t\right); ~ w \in U\left( A\right).  
		$$
		There are $t_0 = 0 < t_1 < ... < t_{n-1} < t_n = 1$ such that $u_{t_{j+1}} u^{-1}_{t_j}\in \W$. If 
		$$
		w = f^{-1}\left(u_{t_{n}} u^{-1}_{t_{n-1}} \right) \cdot ... \cdot f^{-1}\left(u_{t_{1}} u^{-1}_{t_0} \right) \in U\left(M\left( A\right) \right)
		$$
		then $\mu\left(u_1 \right) = w \mu\left(u_0 \right)$.
	\end{proof}
\begin{cor}
	The algebra $A\left[ v\right]$ is invariant with respect to homotopy of  $u$.
\end{cor}
\begin{cor}\label{non_hotop_cor}
	The element $u$ is not homotopy equivalent to $1_{A}$. 
\end{cor}
\begin{proof}
If $u$ is homotopy equivalent to $1_A$ then $\mu\left( u\right)\in A$. It contradicts with the Definition \ref{cyclic_projection_defn}.
\end{proof}

\begin{exm}\label{circ_ext}
	Let $C(u) = C(S^1)$ be the algebra from the Example \ref{circle_fin}. Any element $\varphi \in C(S^1)$ corresponds to a periodic function $\varphi^R \in C_b(\mathbb{R})$. Let $u \in C(S^1)$  such that $u$ corresponds to $u^R: \mathbb{R} \to \mathbb{C}$ given by $u^R(x)=e^{ix}; \ x \in \mathbb{R}$. 
	It is known \cite{karoubi:k} that
	\begin{enumerate}
		\item $u$ is unitary;
		\item The $\mathbb{C}$-linear span of $\{u^n\}_{n\in \mathbb{Z}}$ is dense in $C(S^1)$, i.e. 
		\begin{equation*}
		C(S^1)=C(u);
		\end{equation*}
		\item
		\begin{equation*}
		K^1(C(S^1))= \mathbb{Z}, \text{ and } [u]\in C(S^1) \text{ is a generator of } K^1(C(S^1)).
		\end{equation*}
	\end{enumerate} 
	There is a is a generated by $v$ extension $C(u) \to C(v)$, $u \mapsto v^n$, i.e. $C(v) \approx C(u)[v]$. 
\end{exm}
\begin{exm}\label{k_theory_exm}
Let $A$ be a $C^*$-algebra, and let $u \in U\left(A\right) $ be an unitary element such that
\begin{itemize}
\item $\left[ u\right] \in K_1\left( A\right)$ is an element of infinite order.
\item $\left[ u\right]$ is not divisible, i.e.
$$
\left[ u\right] \ne k x;~ \left| k\right| > 1,~x \in K_1\left( A\right).
$$
\end{itemize}
If $\mu$ is an $n^{\text{th}}$ root of identity and $v=\mu\left(u \right)$ then the from the Lemma \ref{ext_gal_lem} it follows that the triple $\left(A, A[v],\mathbb{Z}_n\right)$ is a finite noncommutative covering projection. 
 
\end{exm}

\section{Noncommutative torus}

Let $A_\theta$ be a noncommutative torus, i.e. universal algebra, generated by unitary elements $u$, $v$ with following relation.
$$
u v = e^{2\pi i \th} vu.
$$
Let $\mu$ is an $n^{\text{th}}$ root of the identity map and $w = \mu\left(u \right)$. There is internal automorphism $f_v$ given by
$$
\varphi_v\left( a\right) = v^*av;~ \forall a \in A_\th.
$$
If $C\left( u\right) \subset A_\theta$ is the generated by $u$ algebra then 
$$
\varphi_v\left(C\left( u\right)\right)= C\left( u\right).
$$
$$
\varphi_v\left(u\right)= e^{2\pi i \th} u
$$
This *-automorphism uniquely defines the automorphism of $W^*$-envelope of $C\left(  u\right)''$. 
From $$f_v\left(w^n \right) = \left(f_v\left(w \right)  \right)^n =  e^{2\pi i \th} u$$ it follows that
$$
f_v\left( w\right) = e^{\frac{2\pi i \th}{n}} \nu\left( u\right)  w
$$
where  $\nu$ is an $n^{\text{th}}$ root of the unity. It follows that
$$
wv = \left(  e^{\frac{2\pi i \th}{n}}  \nu\left( u\right)\right)  vw.
$$
Otherwise

$$
wv = \sum_{j = 0}^{n-1} \left\langle wv,w^j\right\rangle_{\widetilde{A}_\th} w^j =
\sum_{g \in \Z_n} \left( g wvw^{-j}\right)  w^j
$$
so

$$
\sum_{g \in \Z_n} g\left( wvw^{-1}\right) =  e^{\frac{2\pi i \th}{n}}  \nu\left( u\right)
$$

From the definition it follows that
$$
\sum_{g \in \Z_n} g\left( wvw^{-1}\right) \in A_\theta,
$$
so
$$
 e^{\frac{2\pi i \th}{n}}  \nu\left( u\right) \in A_\theta
$$
But if $\nu\left( u \right) \notin \C$ then   $\nu\left( u \right) \notin A_\th$. It follows that there is $k \in \N$, such that
$$
\nu\left(u \right) = e^{\frac{2\pi ik}{n}}
$$
$$
uw =  e^{2\pi \widetilde{\th}}wu; \text{ where } \widetilde{\th} = \frac{(k + \th)}{n}
$$
If $dx$ is the Haar measure on $S^1$ such that
$$
\int_{S^1} 1 dx = 1
$$ 
then
$$
e^{2\pi i \widetilde{\theta}} = \int \frac{\tau_{\th}\left(\mu \right)}{\mu} dx 
$$

\section{Noncommutative quantum $SU(2)$ group and noncommutative covering projections}
 
 \paragraph{} Let $q$ be a real number such that $0<q<1$. 
 A quantum group $SU_q(2)$ is the universal $C^*$-algebra algebra generated by two elements $\al$ and $\beta$ satisfying the following relations:
 $$
 \al^*\al + \beta^*\beta = 1, ~~ \al\al^* + q^2\beta\beta^* =1,
 $$
 $$
 \al\bt - q \bt\al = 0, ~~\al\bt^*-q\bt^*\al = 0,
 $$
 $$
 \bt^*\bt = \bt\bt^*
 $$
 From  $SU_1\left(2 \right)\approx C\left(SU\left(2 \right)  \right)$ it follows that  $SU_q(2)$ can be regarded as a noncommutative deformation of $SU(2)$. 
 It is proven \cite{chakraborty_pal:quantum_su_2} that $K_1\left(SU_q\left( 2\right)  \right) \approx \Z$ and there is an unitary element $\ga_r \in U\left(SU_q\left( 2\right)\right) $ such that $\left[\ga_r \right] \in K_1\left(SU_q\left( 2\right)  \right)$ is a generator of $K_1\left(SU_q\left( 2\right)  \right) \approx \Z$. From the Lemma \ref{ext_gal_lem} it follows than for any $n \in \N$ there is a finite noncommutative covering projection $\left( SU_q\left( 2\right), SU_q\left( 2\right)[v], \Z_n \right)$, where $v^n = \ga_r$.

\end{document}